\documentclass[a4paper, 12pt]{article}
\pdfoutput=1

\usepackage[margin=25mm]{geometry}

\usepackage[utf8]{inputenc}
\usepackage[T1]{fontenc}
\usepackage[british]{babel}
\usepackage[style=british]{csquotes}
\usepackage[en-GB]{datetime2}
\DTMlangsetup[en-GB]{ord=omit}
\usepackage{xcolor}

\usepackage{amsthm, thmtools}
\usepackage{mathtools}
\usepackage{titlesec} % needs to be loaded before hyperref
\usepackage{hyperref}
\hypersetup{
    bookmarksnumbered,
    colorlinks,
    linkcolor={cyan!50!blue},
    citecolor={cyan!50!blue},
    urlcolor={cyan!50!blue}
}
\usepackage[capitalise,nameinlink]{cleveref}

\usepackage{tikz}
\usepackage{tikz-cd}

\usepackage[expansion=false,nopatch=footnote]{microtype}
\usepackage[lining,tabular,scale=.875]{FiraSans}
\usepackage[tt=false,semibold]{libertinus}
\usepackage[libertine,smallerops]{newtxmath}
\usepackage[lining,scaled=.8]{FiraMono}
\usepackage[cal=boondox,frak=euler]{mathalpha}

\DeclareFontEncoding{MDA}{}{}
\DeclareSymbolFont{mathdesignA}{MDA}{mdput}{m}{n}
\SetSymbolFont{mathdesignA}{bold}{MDA}{mdput}{b}{n}
\DeclareSymbolFontAlphabet{\mathbb}{mathdesignA}

\tikzcdset{
    arrow style=tikz,
    arrows={/tikz/line width=.5pt},
    diagrams={>={Straight Barb[scale=0.8]}}
}

\DeclareFontFamily{OMX}{MnSymbolE}{}
\DeclareSymbolFont{MnLargeSymbols}{OMX}{MnSymbolE}{m}{n}
\SetSymbolFont{MnLargeSymbols}{bold}{OMX}{MnSymbolE}{b}{n}
\DeclareFontShape{OMX}{MnSymbolE}{m}{n}{
    <-6>  MnSymbolE5
   <6-7>  MnSymbolE6
   <7-8>  MnSymbolE7
   <8-9>  MnSymbolE8
   <9-10> MnSymbolE9
  <10-12> MnSymbolE10
  <12->   MnSymbolE12
}{}
\DeclareFontShape{OMX}{MnSymbolE}{b}{n}{
    <-6>  MnSymbolE-Bold5
   <6-7>  MnSymbolE-Bold6
   <7-8>  MnSymbolE-Bold7
   <8-9>  MnSymbolE-Bold8
   <9-10> MnSymbolE-Bold9
  <10-12> MnSymbolE-Bold10
  <12->   MnSymbolE-Bold12
}{}
\DeclareMathDelimiter{[}{\mathopen}{MnLargeSymbols}{'000}{MnLargeSymbols}{'000}
\DeclareMathDelimiter{]}{\mathclose}{MnLargeSymbols}{'005}{MnLargeSymbols}{'005}
\DeclareMathDelimiter{\llbr}{\mathopen}{MnLargeSymbols}{'102}{MnLargeSymbols}{'102}
\DeclareMathDelimiter{\rrbr}{\mathclose}{MnLargeSymbols}{'107}{MnLargeSymbols}{'107}

\usepackage{bm}
\usepackage{cases}
\usepackage{accents}

\usepackage{setspace}
\usepackage{needspace}
\usepackage[all]{nowidow}

\usepackage{subdepth} % or \usepackage[low-sup]{subdepth}

\newcommand{\initlengths}{%
    \setlength{\abovedisplayshortskip}{3pt plus 9pt minus 3pt}%
    \setlength{\belowdisplayshortskip}{9pt plus 9pt minus 9pt}%
    \setlength{\abovedisplayskip}{9pt plus 9pt minus 9pt}%
    \setlength{\belowdisplayskip}{9pt plus 9pt minus 9pt}%
    \hfuzz 1pt%
    \tolerance 400% to encourage slightly underfull boxes; default is 200
}

\usepackage{titling}
\pretitle{\vspace{0pt}\setstretch{1.05}\begin{center}\LARGE\libertinusDisplay\fontdimen2\font=0.3em} % fontdimen2 is interword space
\posttitle{\end{center}\vspace{6pt}}

\numberwithin{paragraph}{section}

\newcommand{\parasep}{9pt plus 3pt minus 3pt}
\titleformat{\section}{\Large\libertinusDisplay}{\thesection}{1em}{}
\titleformat{\subsection}{\large\firamedium\boldmath}{\thesubsection}{1em}{}

\usepackage{tocloft}

\renewenvironment{abstract}{%
    \centering\begin{minipage}{.85\textwidth}%
    \setlength{\parindent}{1.5em}%
    \centerline{\firamedium\abstractname}%
    \par\vspace{12pt}%
}{\end{minipage}\par\vspace{9pt}}

\makeatletter
\newcommand*{\@parabookmark}{%
  \pdfbookmark[2]{%
    \theparagraph
    \ifx\@currentlabelname\@empty
    \else
      .\space\@currentlabelname%
    \fi
  }{\theparagraph}
}
\newcommand*{\@thmbookmark}{%
  \pdfbookmark[2]{%
    \theparagraph.\space\thmt@thmname
    \ifx\@currentlabelname\@empty
    \else
      .\space\@currentlabelname%
    \fi
  }{\theparagraph}
}
\newcommand*{\parabookmark}{\@parabookmark}
\newcommand*{\thmbookmark}{\@thmbookmark}

\newcommand*{\resumeparabookmarks}{%
  \renewcommand*{\parabookmark}{\@parabookmark}%
  \renewcommand*{\thmbookmark}{\@thmbookmark}%
}

\declaretheoremstyle[
    spaceabove=\parasep, spacebelow=\parasep,
    postheadspace=.5em,
    postheadhook=\thmbookmark,
    headfont=\normalfont\firamedium\firaproportional,
    headpunct={},
    headformat={\NUMBER.\@\ \NAME.\@\NOTE},
    notefont=\normalfont\firamedium\firaproportional\boldmath,
    notebraces={}{.},
    bodyfont=\itshape,
]{theorem}
\declaretheoremstyle[
    spaceabove=\parasep, spacebelow=\parasep,
    postheadspace=.5em,
    headfont=\normalfont\firamedium\firaproportional,
    headpunct={},
    headformat={\NAME.\@\NOTE},
    notefont=\normalfont\firamedium\firaproportional\boldmath,
    notebraces={}{.},
    bodyfont=\itshape,
]{theorem*}
\declaretheoremstyle[
    spaceabove=\parasep, spacebelow=\parasep,
    postheadspace=.5em,
    postheadhook=\thmbookmark,
    headfont=\normalfont\firamedium\firaproportional,
    headpunct={},
    headformat={\NUMBER.\@\ \NAME.\@\NOTE},
    notefont=\normalfont\firamedium\firaproportional\boldmath,
    notebraces={}{.},
]{definition}
\declaretheoremstyle[
    spaceabove=\parasep, spacebelow=\parasep,
    postheadspace=.5em,
    postheadhook=\parabookmark,
    headfont=\normalfont\firamedium\firaproportional,
    headpunct={},
    headformat={\NUMBER.\@\NOTE},
    notefont=\normalfont\firamedium\firaproportional\boldmath,
    notebraces={}{.},
]{para}

\renewenvironment{proof}[1][\proofname]{\par
    \pushQED{\qed}%
    \normalfont\trivlist
    \item[\hskip\labelsep\firamedium #1\@addpunct{.}]\ignorespaces
}{%
    \popQED\endtrivlist\@endpefalse
}

\expandafter\def\csname equation*@qed\endcsname{\equation@qed}
\makeatother

\declaretheorem[sibling=paragraph, style=para, refname={\S,\S\S}]{para}
\declaretheorem[sibling=paragraph, style=theorem, name=Theorem]{theorem}
\declaretheorem[sibling=paragraph, style=theorem, name=Lemma]{lemma}

\declaretheorem[numbered=no, style=theorem*, name=Theorem]{theorem*}
\declaretheorem[numbered=no, style=theorem*, name=Lemma]{lemma*}

\declaretheorem[sibling=paragraph, style=definition, name=Example]{example}

\declaretheorem[sibling=paragraph, style=definition, name=Remark]{remark}

\numberwithin{equation}{paragraph}

\crefdefaultlabelformat{#2{\upshape#1}#3}

\crefformat{equation}{#2{\upshape(#1)}#3}
\crefrangeformat{equation}{{\upshape#3(#1)#4--#5(#2)#6}}
\crefmultiformat{equation}{#2{\upshape(#1)}#3}{ and #2{\upshape(#1)}#3}{, #2{\upshape(#1)}#3}{, and~#2{\upshape(#1)}#3}

\crefformat{enumi}{#2{\upshape#1}#3}
\crefrangeformat{enumi}{{\upshape#3#1#4--#5#2#6}}
\crefmultiformat{enumi}{#2{\upshape#1}#3}{ and #2{\upshape#1}#3}{, #2{\upshape#1}#3}{, and~#2{\upshape#1}#3}

\crefformat{section}{#2{\upshape\S#1}#3}
\crefrangeformat{section}{{\upshape#3\S\S#1#4--#5#2#6}}
\crefmultiformat{section}{{\upshape#2\S\S#1#3}}{ and {\upshape#2#1#3}}{, {\upshape#2#1#3}}{, and~{\upshape#2#1#3}}
\crefformat{subsection}{#2{\upshape\S#1}#3}
\crefrangeformat{subsection}{{\upshape#3\S\S#1#4--#5#2#6}}
\crefmultiformat{subsection}{{\upshape#2\S\S#1#3}}{ and {\upshape#2#1#3}}{, {\upshape#2#1#3}}{, and~{\upshape#2#1#3}}
\crefformat{para}{#2{\upshape\S#1}#3}
\crefrangeformat{para}{{\upshape#3\S\S#1#4--#5#2#6}}
\crefmultiformat{para}{{\upshape#2\S\S#1#3}}{ and {\upshape#2#1#3}}{, {\upshape#2#1#3}}{, and~{\upshape#2#1#3}}

\crefname{figure}{Figure}{Figures}

\usepackage{enumitem}
\setlist{noitemsep}
\setlist[enumerate]{label=\textnormal{(\roman*)}}

\usepackage[labelsep=period, labelfont={bf}]{caption}
\usepackage[bottom, hang, multiple]{footmisc}
\usepackage{etoolbox}
\makeatletter
\patchcmd{\@makefntext}{\ifFN@hangfoot\bgroup}%
{\ifFN@hangfoot\bgroup\def\@makefnmark{\rlap{\firamedium\firaproportional\@thefnmark}}}{}{}%
\makeatother

\newcommand{\preparebibliography}{
    \phantomsection
    \addcontentsline{toc}{section}{References}
    \sloppy
    \setstretch{1.1}
    \renewcommand*{\bibfont}{\normalfont\small}
}

\usepackage[
    backend=biber,
    style=oxnum,
    giveninits,
    maxbibnames=5,
    maxcitenames=5,
    sorting=nyt
]{biblatex}

\DeclareFieldFormat*{title}{\textit{#1}}
\DeclareFieldFormat*{journaltitle}{#1}
\DeclareFieldFormat[article]{version}{\IfDecimal{#1}{version~}{}#1}
\DefineBibliographyStrings{english}{
  withafterword = {with an appendix by},
}

\DeclareFieldFormat{pages}{%
  \iffieldundef{bookpagination}%
    {#1}%
    {\mkpageprefix[bookpagination]{#1}}%
}

\newcommand{\authorinforule}{\noindent\rule{0.38\textwidth}{0.4pt}}

\newlength{\authorwidth}
\newcommand{\authorinfo}[3]{%
    \setlength{\leftskip}{1.5em}
    \setlength{\parindent}{0em}
    \setstretch{1}
    \par%
    {\small%
    \makebox[\authorwidth][l]{#1}%
    \texttt{#2}%
    \\
    #3.}
    \vspace{6pt}\par
}

\newcommand{\calHom}{{\mathcal{H}\mspace{-5mu}\mathit{om}}}
\newcommand{\calMap}{{\mathcal{M}\mspace{-3mu}\mathit{ap}\mspace{1mu}}}

\newcommand{\git}{{/\mspace{-5mu}/}}

\newcommand{\longsimto}{\mathrel{\overset{\smash{\raisebox{-.8ex}{$\sim$}}\mspace{3mu}}{\longrightarrow}}}

\newcommand{\simto}{\mathrel{\mathchoice
    {\overset{\smash{\raisebox{-.8ex}{$\sim$}}\mspace{3mu}}{\to}}
    {\overset{\smash{\raisebox{-.8ex}{$\sim$}}\mspace{3mu}}{\to}}
    {\overset{\smash{\raisebox{-.6ex}{$\scriptstyle\sim$}}\mspace{3mu}}{\to}}
    {\overset{\smash{\raisebox{-.6ex}{$\scriptscriptstyle\sim$}}\mspace{3mu}}{\to}}
}}

\renewcommand{\geq}{\geqslant}

\title{Proper moduli spaces of orthosymplectic complexes}
\author{Chenjing Bu}
\date{}

\begin{document}

\initlengths

\maketitle

\begin{abstract}
    We construct proper good moduli spaces
for moduli stacks of Bridgeland semistable orthosymplectic complexes
on a complex smooth projective variety,
which we propose as a candidate for
compactifying moduli spaces of principal bundles
for the orthogonal and symplectic groups.
We also prove some results
on good moduli spaces of fixed point stacks and
mapping stacks from finite groupoids.

\end{abstract}

\vspace{12pt}

\section{Introduction}

\begin{para}
    Moduli spaces of coherent sheaves on a smooth projective variety
    are natural compactifications for moduli spaces of vector bundles,
    enabling the study of many flavours of sheaf-counting theories,
    which form a major branch of enumerative geometry.

    However, less is known about compactifications
    of moduli spaces of principal $G$-bundles on a complex smooth projective variety,
    for an arbitrary reductive group~$G$.
    This is not an issue for curves,
    as moduli spaces of semistable principal bundles are already compact,
    as shown by \textcite{ramanathan-1996-i,ramanathan-1996-ii}.
    But in dimension~$\geq$~$2$,
    one needs a good notion of singular principal bundles
    to obtain compact moduli spaces.

    An existing approach is that of \emph{principal $\rho$-sheaves},
    as in \textcite{gomez-sols-2005}
    and \textcite{gomez-herrero-zamora-2024},
    depending on a choice of a representation~$\rho$ of~$G$.

    In this work, we propose a different compactification
    when $G = \mathrm{O}_n$ or $\mathrm{Sp}_{2n}$,
    using \emph{orthosymplectic complexes},
    which are complexes of coherent sheaves
    that are isomorphic to their own derived duals.
    The enumerative geometry of such complexes has been
    studied by the author \cite{bu-osp-dt}
    in the context of motivic Donaldson--Thomas theory,
    where they seem to be more convenient than principal $\rho$-sheaves.
\end{para}

\begin{para}
    Our main result, \cref{thm-osp},
    states that for a smooth projective $\mathbb{C}$-variety~$X$
    and a Bridgeland stability condition~$\tau$ on~$X$ satisfying certain assumptions,
    the moduli stack of $\tau$-semistable orthosymplectic complexes on~$X$
    with a fixed Chern character has a proper good moduli space
    in the sense of \textcite{alper-2013}.

    This is a consequence of a more general result, \cref{thm-proper},
    which implies that if a stack has a proper good moduli space,
    then the same holds for its fixed loci of finite group actions
    and mapping stacks from finite groupoids into it.

    Since the moduli stack of semistable orthosymplectic complexes
    is a $\mathbb{Z} / 2$-fixed locus in the moduli stack
    of semistable complexes,
    and the latter admits proper good moduli spaces by results in
    \cite{alper-halpern-leistner-heinloth-2023-existence},
    \cref{thm-osp} follows.
\end{para}

\begin{para}
    One case that might be of particular interest is when~$X$
    is a K3 surface or an abelian surface.
    In this case, the moduli stack of orthosymplectic complexes
    admits a derived enhancement with a $0$-shifted symplectic structure in the sense of
    \textcite{pantev-toen-vaquie-vezzosi-2013}
    (see \cite[\S 6.2]{bu-osp-dt}),
    and it follows that the proper good moduli space
    admits a Poisson structure
    which restricts to a symplectic structure on the stable locus.

    It would be interesting to determine if these moduli spaces have
    symplectic singularities in the sense of \textcite{beauville-2000}
    (see \textcite{perego-rapagnetta-2023} for the case of sheaves),
    and if they admit symplectic resolutions,
    which, in the case of sheaves,
    are constructed in special cases by \textcite{ogrady-1999}
    and shown to not exist in other cases by
    \textcite{kaledin-lehn-sorger-2006-singular}.
    We expect that our moduli spaces are locally modelled on
    symplectic reductions of moduli of representations of \emph{self-dual quivers}
    in the sense of \textcite{derksen-weyman-2002}
    and \cite[\S 6.1]{bu-osp-dt},
    and that an analysis similar to \cite{kaledin-lehn-sorger-2006-singular}
    would help with answering these questions.
\end{para}

\begin{para}[Acknowledgements]
    The author thanks Fei Peng for
    suggesting a simpler proof of \cref{thm-proper},
    and Dominic Joyce,
    Young-Hoon Kiem,
    Tasuki Kinjo,
    and Naoki Koseki
    for helpful discussions and comments.
    The author also thanks the anonymous referees
    for their useful comments.

    This work was supported by EPSRC grant reference EP/X040674/1.

    % For the purpose of open access,
    % the author has applied a CC-BY public copyright licence to any
    % Author Accepted Manuscript (AAM) version arising from this submission.
\end{para}

\section{Fixed loci and mapping stacks}

In this section, we provide background material and preparatory results
on fixed loci of finite group actions on stacks,
and on mapping stacks from finite groupoids into stacks.

For background on algebraic stacks, see \textcite{olsson-2016} or \textcite{stacks-project}.

\begin{para}[Fixed loci]
    Let~$S$ be an algebraic space,
    and let~$\mathcal{X}$ be an algebraic stack over~$S$.

    If a finite group~$\Gamma$ acts on~$\mathcal{X}$ over~$S$,
    we define its \emph{fixed locus} $\mathcal{X}^\Gamma$
    as the homotopy limit
    \begin{equation*}
        \mathcal{X}^\Gamma =
        \lim_{\mathrm{B} \Gamma} \mathcal{X}
    \end{equation*}
    of the corresponding functor from $\mathrm{B} \Gamma$
    to the $2$-category of algebraic stacks over~$S$.

    More explicitly, it is defined by the functor of points
    $\mathcal{X}^\Gamma (T) = \mathcal{X} (T)^{\Gamma}$,
    where~$T$ is an affine scheme with a morphism $T \to S$,
    and $\mathcal{X} (T)^{\Gamma}$ is the homotopy fixed locus
    of the $\Gamma$-action on the groupoid~$\mathcal{X} (T)$,
    which is the groupoid of pairs $(x, \varphi)$,
    where $x \in \mathcal{X} (T)$,
    and $\varphi = (\varphi_g)_{g \in \Gamma}$
    is a family of isomorphisms
    $\varphi_g \colon g x \simto x$ in $\mathcal{X} (T)$
    satisfying $\varphi_1 = \mathrm{id}_x$ and the cocycle condition
    $\varphi_{gh} = \varphi_g \circ g \varphi_h$ for all $g, h \in \Gamma$.

    See also \textcite{romagny-2005} for related discussions.
\end{para}

\begin{para}[Mapping stacks]
    Let~$\mathcal{X}$ be an algebraic stack over an algebraic space~$S$,
    and let~$\Gamma$ be a finite group.
    The \emph{mapping stack}
    \begin{equation*}
        \calMap (\mathrm{B} \Gamma, \mathcal{X})
    \end{equation*}
    is defined by the functor of points
    $\calMap (\mathrm{B} \Gamma, \mathcal{X}) (T) =
    \mathrm{Map} (\mathrm{B} \Gamma, \mathcal{X} (T))$
    for an affine scheme~$T$ with a morphism $T \to S$,
    where the right-hand side is the mapping groupoid,
    or more explicitly, the groupoid of pairs $(x, \rho)$
    with $x \in \mathcal{X} (T)$
    and $\rho \colon \Gamma \to \mathrm{Aut} (x)$ a group homomorphism.

    In fact, we have
    \begin{equation*}
        \calMap (\mathrm{B} \Gamma, \mathcal{X}) \simeq \mathcal{X}^\Gamma
    \end{equation*}
    for the trivial $\Gamma$-action on~$\mathcal{X}$.
    This is because on the functor of points, the mapping groupoid
    $\mathrm{Map} (-, -)$
    takes colimits to limits in the first component.

    For example, for a finite group~$\Gamma$
    and a flat group algebraic space~$G$
    locally of finite presentation over~$S$,
    we have
    $\calMap (\mathrm{B} \Gamma, \mathrm{B} G) \simeq
    \mathrm{Hom} (\Gamma, G) / G$,
    where $\mathrm{Hom} (\Gamma, G)$
    is the algebraic space of group homomorphisms
    $\rho \colon \Gamma \to G$,
    as a subspace in $G^{|\Gamma|}$ cut out by the relations
    $\rho (1) = 1$ and $\rho (g h) = \rho (g) \rho (h)$ for all $g, h \in \Gamma$,
    and~$G$ acts on $\mathrm{Hom} (\Gamma, G)$ by conjugation.
    This can be checked on the functor of points,
    since a $\Gamma$-action on a $G$-torsor $Y \to T$
    is the same as a homomorphism
    $\Gamma \to \mathrm{Aut} (Y) \simeq \mathrm{Hom}^G (Y, G)$,
    the group of $G$-equivariant maps from~$Y$ to~$G$,
    where~$G$ acts on~$G$ by conjugation,
    which is the same as a $G$-equivariant map
    $Y \to \mathrm{Hom} (\Gamma, G)$.
\end{para}

\begin{para}[Fixed loci and mapping stacks]
    \label{para-fix-map}
    We explain a relation between fixed loci and mapping stacks introduced above.

    Let~$\Gamma$ be a finite group acting on~$\mathcal{X}$ over~$S$.
    The fixed locus~$\mathcal{X}^\Gamma$
    is equivalent to the stack of sections of the morphism
    $\mathcal{X} / \Gamma \to \mathrm{B} \Gamma$,
    that is, we have
    \begin{equation}
        \label{eq-fix-map}
        \mathcal{X}^\Gamma \simeq
        \calMap (\mathrm{B} \Gamma, \mathcal{X} / \Gamma)
        \underset{\calMap (\mathrm{B} \Gamma, \, \mathrm{B} \Gamma)}{\times}
        \{ \mathrm{id} \} \ .
    \end{equation}
    This follows from taking $\Gamma$-fixed points on both sides of
    $\mathcal{X} \simeq (\mathcal{X} / \Gamma) \times_{\mathrm{B} \Gamma} \mathrm{pt}$,
    where $\Gamma$ acts trivially on
    $\mathcal{X} / \Gamma$ and $\mathrm{B} \Gamma$,
    tautologically on the morphism
    $\mathrm{pt} \to \mathrm{B} \Gamma$,
    and trivially on $\mathcal{X} / \Gamma \to \mathrm{B} \Gamma$,
    inducing the given action on the fibre product~$\mathcal{X}$.

    The image of $\{ \mathrm{id} \} \to \calMap (\mathrm{B} \Gamma, \mathrm{B} \Gamma)$
    is isomorphic to $\mathrm{B} \mathrm{Z} (\Gamma)$,
    where $\mathrm{Z} (\Gamma)$ is the centre of~$\Gamma$.
    Thus, if we write
    $\calMap' (\mathrm{B} \Gamma, \mathcal{X} / \Gamma) \subset
    \calMap (\mathrm{B} \Gamma, \mathcal{X} / \Gamma)$
    for the open and closed substack which is
    the preimage of this copy of $\mathrm{B} \mathrm{Z} (\Gamma)$,
    then
    \begin{equation}
        \label{eq-map-prime}
        \calMap' (\mathrm{B} \Gamma, \mathcal{X} / \Gamma) \simeq
        \mathcal{X}^\Gamma \times \mathrm{B} \mathrm{Z} (\Gamma) \ ,
    \end{equation}
    since by \cref{eq-fix-map},
    the left-hand side is a quotient
    $\mathcal{X}^\Gamma / \mathrm{Z} (\Gamma)$,
    and the action is trivial, which can be seen as follows.
    It is induced by the $(\Gamma \times \mathrm{Z} (\Gamma))$-action on~$\mathcal{X}$
    via the homomorphism
    $\Gamma \times \mathrm{Z} (\Gamma) \to \Gamma$,
    $(g, z) \mapsto g z$,
    which also gives a $\mathrm{Z} (\Gamma)$-action on the
    $\mathrm{B} \Gamma$-shaped diagram encoding the $\Gamma$-action on~$\mathcal{X}$.
    This action on the diagram is compatible with
    the trivial $\mathrm{Z} (\Gamma)$-action on~$\mathcal{X}^\Gamma$,
    in that they extend to a $\mathrm{Z} (\Gamma)$-action
    on the diagram witnessing the limit.
    This compatibility implies that
    the trivial $\mathrm{Z} (\Gamma)$-action on~$\mathcal{X}^\Gamma$
    agrees with the induced action.
\end{para}

\begin{example}
    \label{eg-bgl}
    We present a simple example demonstrating the above relations,
    which also illustrates how higher coherence is involved
    in group actions and fixed loci on stacks.

    Let $S = \operatorname{Spec} \mathbb{C}$.
    Let $\Gamma = \mathbb{Z} / 2$ act on $\mathcal{X} = \mathrm{BGL}_n$
    by taking the inverse transpose matrix,
    and use the conjugation by $\varepsilon \in \{ \pm 1 \} \subset \mathrm{GL}_n$
    to witness the equality of the double inverse transpose of a matrix with itself.
    In other words, regarding $\mathrm{BGL}_n$ as classifying rank~$n$ vector bundles,
    the action is given by taking the dual bundle,
    and we identify the double dual with the original bundle
    using $\varepsilon$~times the usual isomorphism.
    See \cite[Example~2.2.6]{bu-osp-dt} for more on this example.

    The fixed locus of this action is
    \begin{equation*}
        \mathrm{BGL}_n^{\smash{\mathbb{Z} / 2}} \simeq
        \begin{cases}
            \mathrm{BO}_n & \text{if $\varepsilon = 1$,}
            \\[-.5ex]
            \mathrm{BSp}_n & \text{if $\varepsilon = -1$ and $n$ is even,}
            \\[-.5ex]
            \varnothing & \text{if $\varepsilon = -1$ and $n$ is odd.}
        \end{cases}
    \end{equation*}
    On the other hand, we have
    $\mathrm{BGL}_n / (\mathbb{Z} / 2) \simeq \mathrm{B} G$,
    where $\mathrm{GL}_n \hookrightarrow G \twoheadrightarrow \mathbb{Z} / 2$
    is a group extension defined as follows:
    $G$ is generated by $\mathrm{GL}_n \subset G$
    and a generator $j \in G$,
    with the relations $j^2 = \varepsilon$
    and $j g = g^{-\mathrm{t}} j$ for $g \in \mathrm{GL}_n$,
    where $g^{-\mathrm{t}}$ is the inverse transpose.
    Thus, we have
    \begin{equation*}
        \calMap' (\mathrm{B} (\mathbb{Z} / 2), \mathrm{B} G) \simeq H / G \ ,
    \end{equation*}
    where $\calMap'$ is as in \cref{eq-map-prime},
    and $H \subset G \setminus \mathrm{GL}_n$
    is the closed subscheme of elements of order~$2$.
    Since $(j g)^2 = 1$ if and only if $g = \varepsilon g^{\mathrm{t}}$
    for $g \in \mathrm{GL}_n$,
    $H$ is isomorphic to the space of
    \mbox{(anti-)}\allowbreak symmetric invertible matrices when
    $\varepsilon = 1$ (resp.\ $-1$).
    Under this identification,
    the $G$-action on~$H$ is given by
    $g \cdot h = g^{\mathrm{t}} h g$
    and $jg \cdot h = g^{\mathrm{t}} h^{-\mathrm{t}} g$
    for $g \in \mathrm{GL}_n$ and an \mbox{(anti-)}\allowbreak symmetric matrix~$h$.

    When $\varepsilon = 1$,
    $H$ consists of a single $G$-orbit represented by the identity matrix $h = 1$,
    and the stabilizer consists of $g, jg \in G$ for $g \in \mathrm{O}_n$.
    As $j g = g j$ for such~$g$, we have
    $H / G \simeq \mathrm{B} (\mathrm{O}_n \times \mathbb{Z} / 2)$,
    consistent with the prediction in~\cref{para-fix-map} that
    $H / G \simeq \mathrm{BGL}_n^{\smash{\mathbb{Z} / 2}} \times \mathrm{B} (\mathbb{Z} / 2)$.

    When $\varepsilon = -1$,
    $H$ is empty when $n$ is odd.
    When $n$ is even, it again consists of a single $G$-orbit,
    represented by the matrix
    $h = J = \bigl( \begin{smallmatrix} 0 & -1 \\ 1 & 0 \end{smallmatrix} \bigr)$
    of the standard symplectic form.
    The stabilizer consists of $g, jg \in G$ for $g \in \mathrm{Sp}_n$,
    and is isomorphic to $\mathrm{Sp}_n \times \mathbb{Z} / 2$,
    with $jJ$ corresponding to the generator of $\mathbb{Z} / 2$.
    Again, $H / G \simeq \mathrm{B} (\mathrm{Sp}_n \times \mathbb{Z} / 2)
    \simeq \mathrm{BGL}_n^{\smash{\mathbb{Z} / 2}} \times \mathrm{B} (\mathbb{Z} / 2)$.
\end{example}

\begin{theorem}
    \label{thm-map}
    Let~$S$ be an algebraic space,
    and let~$\mathcal{X}$ be an algebraic stack over~$S$
    acted on by a finite group~$\Gamma$.
    Consider the forgetful morphism
    \begin{equation*}
        p \colon \mathcal{X}^\Gamma \longrightarrow \mathcal{X} \ .
    \end{equation*}
    Then $p$~is representable by algebraic spaces and locally of finite type,
    and $\mathcal{X}^\Gamma$ is an algebraic stack.
    Moreover, if\/~$\mathcal{X}$ is quasi-separated,
    then~$p$ is quasi-compact and quasi-separated;
    if\/~$\mathcal{X}$ has affine diagonal over~$S$,
    then~$p$ is affine.

    In particular, the same is true for the mapping stack
    $\calMap (\mathrm{B} \Gamma, \mathcal{X})$
    and the forgetful morphism
    \begin{equation*}
        \calMap (\mathrm{B} \Gamma, \mathcal{X}) \longrightarrow \mathcal{X} \ .
    \end{equation*}
\end{theorem}

\begin{proof}
    The properties of~$p$ follow from
    \textcite[Proposition~3.7]{romagny-2005}.
    Although it is assumed there (following \cite{laumon-moret-bailly-2000})
    that the diagonal of~$\mathcal{X}$ is separated and quasi-compact,
    the proof applies in general,
    except that without the diagonal of~$\mathcal{X}$ being separated or quasi-compact,
    it is not guaranteed that~$p$ has the same properties.
    The algebraicity of~$\mathcal{X}^\Gamma$ follows from the representability of~$p$.
    Note that the diagonal of~$\mathcal{X}$ is always
    representable by algebraic spaces and locally of finite type
    (see \cite[Tag~\href{https://stacks.math.columbia.edu/tag/04XS}{\texttt{04XS}}]{stacks-project}).

    The final statement follows from
    considering the trivial $\Gamma$-action on~$\mathcal{X}$.
\end{proof}

\section{Good moduli spaces}

This section is dedicated to proving the following result.
See \textcite{alper-2013} for background on good moduli spaces.

\begin{theorem}
    \label{thm-proper}
    Let~$S$ be a noetherian algebraic space of characteristic zero,
    and let~$\mathcal{X}$ be an algebraic stack locally of finite type with affine diagonal over~$S$,
    with a good moduli space $\mathcal{X} \to X$.

    Then for any finite group~$\Gamma$,
    the stack $\calMap (\mathrm{B} \Gamma, \mathcal{X})$
    has a good moduli space which is finite over~$X$.
    Moreover, if\/~$X$ is separated,
    then for any\/~$\Gamma$-action on~$\mathcal{X}$ over~$S$,
    the fixed locus $\mathcal{X}^\Gamma$ also has a good moduli space
    which is finite over~$X$.
\end{theorem}

\begin{proof}
    We first prove the statement about
    $\calMap (\mathrm{B} \Gamma, \mathcal{X})$.

    By \cref{thm-map}, the morphism
    $\calMap (\mathrm{B} \Gamma, \mathcal{X}) \to \mathcal{X}$
    is affine and of finite type, and in particular,
    $\calMap (\mathrm{B} \Gamma, \mathcal{X})$ has affine diagonal over~$S$.
    By \textcite[Lemma~4.14]{alper-2013},
    $\calMap (\mathrm{B} \Gamma, \mathcal{X})$
    has a good moduli space~$Y$ which is affine and of finite type over~$X$,
    and it is enough to show that~$Y$ is also universally closed over~$X$.
    By \cite[Theorem~4.16~(i)]{alper-2013},
    $\calMap (\mathrm{B} \Gamma, \mathcal{X}) \to Y$
    is surjective, and it is enough to show that
    $\calMap (\mathrm{B} \Gamma, \mathcal{X}) \to X$
    is universally closed.

    Since $\mathcal{X} \to X$ is universally closed
    by \cite[Theorem~4.16~(ii)]{alper-2013},
    it is enough to prove the following statement:%
    \footnote{
        See
        \cite[Tag~\href{https://stacks.math.columbia.edu/tag/0CLV}{\texttt{0CLV}}]{stacks-project}
        for the valuative criterion for universal closedness for stacks.
    }
    Let $R$ be a discrete valuation ring with fraction field~$K$,
    let $x \in \mathcal{X} (R)$ be an $R$-point,
    and let $y \in \calMap (\mathrm{B} \Gamma, \mathcal{X}) (K)$
    be a $K$-point whose image in~$\mathcal{X}$ is $x_K$.
    Then there is an extension of discrete valuation rings $R \to R'$,
    such that $y_{K'}$ (where $K' = \mathrm{Frac} (R')$)
    extends to an $R'$-point
    $y' \in \calMap (\mathrm{B} \Gamma, \mathcal{X}) (R')$
    lying over the same $R'$-point of~$X$ as~$x_{R'}$.

    The groupoid
    $\calMap (\mathrm{B} \Gamma, \mathcal{X}) (K)$
    is equivalent to the groupoid of pairs $(z, \rho)$,
    where $z \in \mathcal{X} (K)$ and
    $\rho \colon \Gamma \to \mathrm{Aut} (z)$ is a group homomorphism.
    Fix $(z, \rho)$ to be the pair attached to~$y$,
    so $z = x_K$.
    Applying \cref{lem-ahlh-fin-gp},
    we see that there exists an extension $R \to R'$
    and an $R'$-point $x' \in \mathcal{X} (R')$
    with $x'_{\smash{K'}} = z$,
    such that~$\rho$ lifts to a homomorphism
    $\rho' \colon \Gamma \to \mathrm{Aut} (x')$.
    The pair $(x', \rho')$ now defines an $R'$-point
    $y' \in \calMap (\mathrm{B} \Gamma, \mathcal{X}) (R')$
    extending~$y_{\smash{K'}}$, as desired.

    Finally, for the last statement,
    we apply the above to the quotient stack~$\mathcal{X} / \Gamma$.
    Since~$X$ is separated,
    the quotient stack $X / \Gamma$ has finite inertia,
    and has a good moduli space $X \git \Gamma$.
    It is also the good moduli space of~$\mathcal{X} / \Gamma$, as
    $\mathcal{X} / \Gamma \to X / \Gamma \to X \git \Gamma$
    is a composition of good moduli space morphisms in the sense of
    \cite[Remark~4.4]{alper-2013}.
    The previous part applied to
    $\calMap (\mathrm{B} \Gamma, \mathcal{X} / \Gamma)$
    now gives a good moduli space of
    $\mathcal{X}^\Gamma \times \mathrm{B} \mathrm{Z} (\Gamma)$,
    which is also a good moduli space of~$\mathcal{X}^\Gamma$,
    and it is finite over~$X \git \Gamma$,
    hence finite over~$X$.
\end{proof}

\begin{remark}
    In the last part of \cref{thm-proper},
    the assumption that~$X$ is separated cannot be dropped.
    For example, if
    $\mathcal{X} = X =
    \mathbb{A}^1 \cup_{\smash{\mathbb{A}^1 \setminus \{ 0 \}}} \mathbb{A}^1$
    is the line with two origins,
    and $\Gamma = \mathbb{Z} / 2$ acts by swapping the two origins,
    then the fixed locus $\mathbb{A}^1 \setminus \{ 0 \}$
    is not finite over~$X$.
    In this case, $X / \Gamma$
    does not have a good moduli space,
    so the first part of the theorem does not apply to it.
\end{remark}

\begin{lemma}
    \label{lem-ahlh-fin-gp}
    Let~$S$ be a noetherian algebraic space of characteristic zero,
    and let~$\mathcal{X}$ be an algebraic stack
    locally of finite type with affine diagonal over~$S$,
    with a good moduli space $\mathcal{X} \to X$.

    Then given a discrete valuation ring~$R$,
    an $R$-point $x \in \mathcal{X} (R)$,
    a finite group~$\Gamma$,
    and a group homomorphism~$\rho \colon \Gamma \to \mathrm{Aut} (x_K)$,
    where $K = \mathrm{Frac} (R)$,
    there exists an extension of discrete valuation rings $R \to R'$
    and an $R'$-point $x' \in \mathcal{X} (R')$
    lying over the same $R'$-point of\/~$X$ as~$x_{\smash{R'}}$,
    with $x'_{\smash{K'}} \simeq x_{\smash{K'}}$,
    where $K' = \mathrm{Frac} (R')$,
    such that
    $\rho_{\smash{K'}} \colon \Gamma \to \mathrm{Aut} (x_{\smash{K'}})$
    extends to a $\Gamma$-action on~$x'$.
\end{lemma}

\begin{proof}
    When~$\Gamma$ is a finite cyclic group, this follows from
    \textcite[Theorem~5.3~(1)]{alper-halpern-leistner-heinloth-2023-existence},
    which implies that~$\mathcal{X}$ satisfies the property stated in
    \cite[Proposition~5.2~(1)]{alper-halpern-leistner-heinloth-2023-existence},%
    \footnote{
        Although the property
        is only stated for discrete valuation rings essentially of finite type over~$S$,
        it is proved for all discrete valuation rings there.
    }
    where we use~$X$ as the base.

    The general case also follows from the proof of
    \cite[Theorem~5.3~(1)]{alper-halpern-leistner-heinloth-2023-existence},
    but with the following modifications.

    Firstly, we prove
    \cite[Lemma~5.12]{alper-halpern-leistner-heinloth-2023-existence}
    in this greater generality:
    We show that for $R, K$ as above and an integer $N \geq 0$,
    every group homomorphism $\Gamma \to \mathrm{GL}_N (K)$
    is conjugate to a homomorphism
    $\Gamma \to \mathrm{GL}_N (R) \subset \mathrm{GL}_N (K)$.
    Equivalently, for any $\Gamma$-representation~$V \simeq K^N$,
    there exists a $\Gamma$-invariant $R$-submodule $V_R \simeq R^N \subset V$ spanning~$V$.

    To prove this, take any submodule $V_0 \simeq R^N \subset V$ spanning~$V$,
    and let $V_R = \sum_{\gamma \in \Gamma} \gamma V_0$.
    Then $V_R$ is a $\Gamma$-invariant finitely generated
    torsion-free $R$-submodule of~$V$, hence free.
    The natural map $V_R \otimes_R K \to V$ is an isomorphism:
    It is surjective by construction,
    and injective since $(-) \otimes_R K$ preserves injections
    and $V \otimes_R K \simeq V$.
    This proves the generalized lemma.

    Secondly, the proof of
    \cite[Proposition~5.11]{alper-halpern-leistner-heinloth-2023-existence}
    needs to be modified as follows.
    It starts with an affine scheme $U = \operatorname{Spec} A$
    with a $\mathrm{GL}_N$-action,
    and constructs a scheme $X$ containing~$U$
    and projective over $U \git \mathrm{GL}_N$.
    The stabilizer scheme
    $\mathrm{Stab}_{\mathrm{GL}_N} (X) \subset X \times \mathrm{GL}_N$
    parametrizing pairs $(x, g)$ with $x \in X$ and $g \in \mathrm{GL}_N$
    with $g \cdot x = x$,
    which is projective over
    $U \git \mathrm{GL}_N \times \mathrm{GL}_N$,
    and has a $\mathrm{GL}_N$-action
    by the given action on~$X$ and conjugation on~$\mathrm{GL}_N$.
    For a certain GIT stability condition on
    $\mathrm{Stab}_{\mathrm{GL}_N} (X)$,
    its semistable locus is
    $\mathrm{Stab}_{\mathrm{GL}_N} (U)$,
    and by general results in GIT, the morphism
    $\mathrm{Stab}_{\mathrm{GL}_N} (U) \git \mathrm{GL}_N \to
    U \git \mathrm{GL}_N \times \mathrm{GL}_N \git \mathrm{GL}_N$
    is proper.
    It follows that
    $\mathrm{Stab}_{\mathrm{GL}_N} (U) / \mathrm{GL}_N
    \to U \git \mathrm{GL}_N \times \mathrm{GL}_N \git \mathrm{GL}_N$
    is universally closed,
    so that a morphism
    $\operatorname{Spec} R \to \mathrm{Stab}_{\mathrm{GL}_N} (X) / \mathrm{GL}_N$
    mapping the generic point to the semistable part
    can be modified to a morphism
    $\operatorname{Spec} R' \to \mathrm{Stab}_{\mathrm{GL}_N} (U) / \mathrm{GL}_N$
    entirely inside the semistable part.

    For our purpose, we need to replace $\mathrm{Stab}_{\mathrm{GL}_N} (X)$
    by the scheme
    $\mathrm{Stab}_{\mathrm{GL}_N}^{\Gamma} (X) \subset X \times \mathrm{GL}_N^\Gamma$
    parametrizing pairs $(x, \rho)$ with $x \in X$
    and $\rho \colon \Gamma \to \mathrm{GL}_N$ a group homomorphism
    such that $g (\gamma) \cdot x = x$ for all $\gamma \in \Gamma$;
    here $\mathrm{GL}_N^\Gamma$ denotes the product
    of $|\Gamma|$ copies of~$\mathrm{GL}_N$.
    Then $\mathrm{GL}_N$ acts on $\mathrm{Stab}_{\mathrm{GL}_N}^{\Gamma} (X)$
    by the given action on~$X$ and conjugation on~$\rho$,
    and $\mathrm{Stab}_{\mathrm{GL}_N}^{\Gamma} (X)$ is projective
    over $U \git \mathrm{GL}_N \times \mathrm{GL}_N^\Gamma$.
    The stability condition on~$\mathrm{Stab}_{\mathrm{GL}_N}^{\Gamma} (X)$
    is given by the pullback of a line bundle on~$X$ as in the original proof,
    so the semistable locus is $\mathrm{Stab}_{\mathrm{GL}_N}^{\Gamma} (U)$,
    and we obtain a proper morphism
    $\mathrm{Stab}_{\mathrm{GL}_N}^{\Gamma} (U) \git \mathrm{GL}_N \to
    U \git \mathrm{GL}_N \times \mathrm{GL}_N^\Gamma \git \mathrm{GL}_N$.
    As in the original proof,
    any morphism $\operatorname{Spec} R \to \mathrm{Stab}_{\mathrm{GL}_N}^{\Gamma} (X) / \mathrm{GL}_N$
    can be modified to a morphism
    $\operatorname{Spec} R' \to \mathrm{Stab}_{\mathrm{GL}_N}^{\Gamma} (U) / \mathrm{GL}_N$.

    The proof of \cite[Theorem~5.3~(1)]{alper-halpern-leistner-heinloth-2023-existence}
    can now be followed with the above modifications.
\end{proof}

\section{Orthosymplectic complexes}

In this section, we apply \cref{thm-proper}
to construct proper moduli spaces for orthosymplectic complexes.
We first recall the notion of orthosymplectic complexes from
\textcite[\S 6.2]{bu-osp-dt}.

\begin{para}[Orthosymplectic complexes]
    Let~$X$ be a connected smooth projective $\mathbb{C}$-variety,
    and we fix the data
    $(I, L, s, \varepsilon)$, where
    $I \colon X \simto X$ is an involution,
    $L \to X$ is a line bundle,
    $s$ is an integer, and
    $\varepsilon \colon L \simto I^* (L)$ is an isomorphism
    such that $I^* (\varepsilon) \circ \varepsilon = \mathrm{id}_L$.

    This data defines a \emph{self-dual structure}
    on the dg-category $\mathsf{Perf} (X)$
    of perfect complexes on~$X$,
    that is a contravariant involution
    \begin{equation*}
        \mathbb{D} = \mathbb{R} \calHom (I^* (-), L) [s]
        \colon \mathsf{Perf} (X) \longsimto \mathsf{Perf} (X)^\mathrm{op} \ ,
    \end{equation*}
    with the identification
    $\mathbb{D}^2 \simeq \mathrm{id}$ induced by the isomorphism~$\varepsilon$.

    A \emph{self-dual complex} is defined as a fixed point
    of the $\mathbb{Z} / 2$-action on the underlying groupoid $\mathsf{Perf} (X)^\simeq$.
    Explicitly, it is the data of a complex~$E \in \mathsf{Perf} (X)$,
    an isomorphism $\varphi \colon E \simto \mathbb{D} (E)$,
    an equivalence $\varphi \simto \mathbb{D} (\varphi)$,
    and higher coherence data.
    If~$E$ lies in a heart of some t-structure on $\mathsf{Perf} (X)$,
    then the data $(E, \varphi)$ is sufficient,
    with no need for higher coherence.

    For example, a natural choice is to take
    $I = \mathrm{id}_X$, $L = \mathcal{O}_X$, $s = 0$, and
    $\varepsilon = \mathrm{id}_{\mathcal{O}_X}$
    (resp.~$-\mathrm{id}_{\mathcal{O}_X}$).
    In this case, orthogonal (resp.~symplectic) vector bundles on~$X$
    are self-dual complexes.
    For this reason, we call the self-dual complexes
    \emph{orthogonal} (resp.~\emph{symplectic}) \emph{complexes}
    in this case.
\end{para}

\begin{para}[Stability conditions]
    Let~$(Z, \mathcal{P})$ be a Bridgeland stability condition on~$X$,
    as in \cite{bridgeland-2007-stability},
    where $Z \colon K (X) \to \mathbb{C}$ is the central charge,
    where we set $K (X) \subset \mathrm{H}^{2 \bullet} (X; \mathbb{Q})$
    to be the subset consisting of Chern characters of perfect complexes,
    and $\mathcal{P} = (\mathcal{P} (t))_{t \in \mathbb{R}}$
    is a slicing of $\mathsf{Perf} (X)$.

    The \emph{dual stability condition}~$(Z^\vee, \mathcal{P}^\vee)$
    of~$(Z, \mathcal{P})$ is defined by
    $Z^\vee (\alpha) = \overline{Z (\mathbb{D} \alpha)}$
    for all $\alpha \in K (X)$,
    where the bar denotes complex conjugation,
    and $\mathcal{P}^\vee (t) = \mathbb{D} (\mathcal{P} (-t))$
    for all $t \in \mathbb{R}$.

    For the stability condition to be applied to self-dual complexes,
    we require it to be \emph{self-dual},
    that is, we require that
    $(Z, \mathcal{P}) = (Z^\vee, \mathcal{P}^\vee)$.

    Furthermore, to ensure good behaviour of moduli stacks,
    we require the stability condition $(Z, \mathcal{P})$
    to have \emph{rational central charge},
    that is, $Z (K (X)) \subset \mathbb{Q} + \mathrm{i} \mathbb{Q}$,
    and to satisfy the \emph{generic flatness} and \emph{boundedness} properties,
    as in \textcite[\S 4]{piyaratne-toda-2019} or
    \textcite[Definitions~6.2.4 and~6.5.2]{halpern-leistner-instability}.
    See \cite[Example~6.2.5]{bu-osp-dt}
    for examples of such stability conditions
    when~$X$ is either a curve, a surface,
    or a threefold satisfying the Bogomolov--Gieseker inequality of
    \textcite{bayer-macri-toda-2014}.
\end{para}

\begin{para}[Moduli stacks]
    Given a self-dual stability condition $\tau = (Z, \mathcal{P})$ as above,
    the dual functor~$\mathbb{D}$ preserves
    semistable objects of slope~$0$,
    and the moduli stack $\mathcal{X}^{\mathrm{ss}} (0; \tau)$
    of semistable objects of slope~$0$ (including the zero object)
    admits an induced $\mathbb{Z} / 2$-action,
    similar to the one considered in \cref{eg-bgl}.
    The fixed locus
    \begin{equation*}
        \mathcal{X}^{\mathrm{sd,ss}} (\tau) =
        \mathcal{X}^{\mathrm{ss}} (0; \tau)^{\mathbb{Z} / 2}
    \end{equation*}
    is the moduli of semistable self-dual complexes.

    By \textcite[Lemma~7.20, Examples~7.26 and~7.29]{alper-halpern-leistner-heinloth-2023-existence},
    the stack $\mathcal{X}^{\mathrm{ss}} (0; \tau)$ has affine diagonal,
    and for every Chern character $\alpha \in K (X)$ with $Z (\alpha) \in \mathbb{R}_{> 0}$,
    the open and closed substack
    $\mathcal{X}^{\mathrm{ss}}_\alpha (0; \tau) \subset \mathcal{X}^{\mathrm{ss}} (0; \tau)$
    parametrizing semistable objects with Chern character~$\alpha$ and slope~$0$
    has a proper good moduli space.
    Therefore, applying \cref{thm-proper}, we obtain the following:
\end{para}

\begin{theorem}
    \label{thm-osp}
    Let~$X$ be a connected smooth projective $\mathbb{C}$-variety,
    and fix the data~$(I, L, s, \varepsilon)$ as above.
    Let~$\tau$ be a self-dual Bridgeland stability condition on~$X$
    with rational central charge
    and satisfying the generic flatness and boundedness properties.

    Then for every Chern character $\alpha \in K (X)$
    with $\alpha = \mathbb{D} (\alpha)$ and $Z (\alpha) \in \mathbb{R}_{> 0}$,
    the open and closed substack
    \begin{equation*}
        \mathcal{X}^{\mathrm{sd,ss}}_\alpha (\tau) \subset \mathcal{X}^{\mathrm{sd,ss}} (\tau)
    \end{equation*}
    parametrizing semistable self-dual complexes with Chern character~$\alpha$
    has a proper good moduli space.
    In particular, every connected component of
    $\mathcal{X}^{\mathrm{sd,ss}} (\tau)$
    has a proper good moduli space.
    \qed
\end{theorem}

\begin{remark}
    As in \cite[Example~7.29]{alper-halpern-leistner-heinloth-2023-existence},
    the theorem also holds for any self-dual Bridgeland stability condition~$\tau'$
    that is in the same connected component in the space of
    (not necessarily self-dual) stability conditions
    as one that satisfies the above properties,
    because for each fixed class $\alpha \in K (X)$,
    one can always find a stability~$\tau$ satisfying the above properties
    that is close enough to~$\tau'$ and defines the same semistable locus for~$\alpha$.
\end{remark}

% \clearpage
\preparebibliography
\printbibliography

\authorinforule

\authorinfo{Chenjing Bu}
    {bucj@mailbox.org}
    {Mathematical Institute, University of Oxford, Oxford OX2 6GG, United Kingdom}

\end{document}